\documentclass[leqno,12pt]{article} 
\usepackage{amsmath, amssymb}
\usepackage{amsthm} 
\theoremstyle{plain} 
\newtheorem{theorem}{\indent\sc Theorem}[section]

\newtheorem{proposition}[theorem]{\indent\sc Proposition}

\theoremstyle{definition} 

\newtheorem{remark}[theorem]{\indent\sc Remark}
\newtheorem{example}[theorem]{\indent\sc Example}

\title{Remarks on almost $\eta$-Ricci solitons \\
in $\left( \varepsilon \right) $-para Sasakian manifolds} 

\author{Adara M. Blaga and Selcen Y\"{u}ksel Perkta\c{s}}

\date{} 
\begin{document}

\maketitle


\begin{abstract}
We consider almost $\eta$-Ricci solitons in $\left( \varepsilon \right) $-para Sasakian manifolds satisfying certain curvature conditions. In the gradient case we give an estimation of the Ricci curvature tensor's norm and express the scalar curvature of the manifold in terms of the functions that define the soliton. We also prove that if the Ricci operator is Codazzi, then the gradient $\eta$-Ricci soliton is expanding if $M$ is spacelike or shrinking if $M$ is timelike.
\end{abstract}

\section{Introduction}

A pseudo-Riemannian
manifold $(M,g)$ admits a \textit{Ricci soliton} \cite{Hamilton-1982} if there exists a smooth vector
field $V$ (called the potential vector field) such that
\begin{equation}
\frac{1}{2}\pounds _{V}\,g+S+\lambda g=0,  \label{int-1}
\end{equation}%
where $\pounds _{V}$ denotes the Lie derivative along the vector field $V$ and $%
\lambda $ is a real constant. By perturbing (\ref{int-1}) with a term $\mu \eta\otimes \eta$, for $\mu$ a real constant and $\eta$ a $1$-form, we obtain the \textit{$\eta$-Ricci soliton}, introduced by Cho and Kimura \cite%
{Cho-Kimura} and more general, if we replace the two constants $\lambda$ and $\mu$ by smooth functions, we get \textit{almost $\eta$-Ricci solitons} \cite{bla}.

In the present paper we consider almost $\eta$-Ricci solitons on $\left( \varepsilon \right) $-para Sasakian manifolds which satisfy certain curvature properties. \textit{$\left( \varepsilon \right) $-para Sasakian manifolds} were defined by Tripathi, K\i l\i \c{c}, Y\"{u}ksel Perkta%
\c{s}, and Kele\c{s} \cite{Tri-KYK-10} as a conterpart of almost paracontact metric geometry \cite{Sato-76}.
Our interest is also to characterize the geometry of an almost $\eta$-Ricci soliton on $\left( \varepsilon \right) $-para Sasakian manifolds in the case when the potential vector field is of gradient type and explicitly compute the scalar curvature for the case of gradient $\eta$-Ricci solitons. Remark that some properties of $\eta$-Ricci solitons on $\left( \varepsilon \right) $-almost paracontact metric manifolds were studied in \cite{per}.

\section{$\left( \varepsilon \right) $-para Sasakian structures}

Recall that an \textit{almost paracontact structure} on an $n$-dimensional manifold $M$ is a triple $(\varphi ,\xi ,\eta
)$ \cite{Sato-76} consisting of a $(1,1)$-tensor field $\varphi $, a vector field $\xi $ and a $1$-form $\eta $ satisfying:
\begin{equation}
\varphi ^{2}=I-\eta \otimes \xi ,  \label{eq-phi-eta-xi}
\end{equation}%
\begin{equation}
\eta (\xi )=1,  \label{eq-eta-xi}
\end{equation}%
\begin{equation}
\varphi \xi =0,  \label{eq-phi-xi}
\end{equation}%
\begin{equation}
\eta \circ \varphi =0.  \label{eq-eta-phi}
\end{equation}%
One can easily check that (\ref{eq-phi-eta-xi}) and one of (\ref{eq-eta-xi}%
), (\ref{eq-phi-xi}) and (\ref{eq-eta-phi}) imply the other two relations.
Moreover, if $g$ is a pseudo-Riemannian metric such that
\begin{equation}
g\left( \varphi \cdot,\varphi \cdot\right) =g\left( \cdot,\cdot\right) -\varepsilon \eta
\otimes\eta,  \label{eq-metric-1}
\end{equation}%
where $\varepsilon =\pm 1$, then $M$ is called $\left( \varepsilon
\right) ${\it -almost paracontact metric manifold} equipped with an \textit{$\left(
\varepsilon \right) ${\em -}almost paracontact metric structure} $(\varphi
,\xi ,\eta ,g,\varepsilon )$ \cite{Tri-KYK-10}. In particular, if ${\rm index%
}(g)=1$ (that is when $g$ is a Lorentzian metric), then the $(\varepsilon )$%
-almost paracontact metric manifold is called {\it Lorentzian almost
paracontact manifold}. From (\ref{eq-metric-1}) we obtain:
\begin{equation}
i_{\xi}g=\varepsilon \eta,  \label{eq-metric-3}
\end{equation}%
\begin{equation}
g\left( X,\varphi Y\right) =g\left( \varphi X,Y\right),  \label{eq-metric-2}
\end{equation}%
and from (\ref{eq-metric-3}) it follows that
\begin{equation}
g\left( \xi ,\xi \right) =\varepsilon,  \label{eq-g(xi,xi)}
\end{equation}%
that is, the structure vector field $\xi $ is never lightlike.

\bigskip

Let $(M,\varphi ,\xi ,\eta ,g,\varepsilon )$ be an $(\varepsilon )$-almost
paracontact metric manifold (resp. a Lorentzian almost paracontact
manifold). If $\varepsilon =1$, then $M$ is said to be a spacelike $%
(\varepsilon )$-almost paracontact metric manifold (resp. a spacelike
Lorentzian almost paracontact manifold) and if $\varepsilon =-\,1$,
then $M$ is said to be a timelike $(\varepsilon )$-almost paracontact
metric manifold (resp. a timelike Lorentzian almost paracontact manifold)
\cite{Tri-KYK-10}.

\bigskip

An $\left( \varepsilon \right) $-almost paracontact metric structure $(\varphi ,\xi ,\eta ,g,\varepsilon )$ is
called $\left( \varepsilon \right) ${\it -para Sasakian structure} if
\begin{equation}
(\nabla _{X}\varphi )Y=-\,g(\varphi X,\varphi Y)\xi -\varepsilon \eta \left(
Y\right) \varphi ^{2}X,  \label{para2}
\end{equation}%
for any $X,Y\in \Gamma (TM)$, where $\nabla $ is the Levi-Civita connection with respect to $g$. A
manifold endowed with an $\left( \varepsilon \right) $-para Sasakian
structure is called $\left( \varepsilon \right) ${\it -para Sasakian
manifold} \cite{Tri-KYK-10}. In an $\left( \varepsilon \right) ${\em -}para
Sasakian manifold, we have
\begin{equation}
\nabla \xi =\varepsilon \varphi  \label{para3}
\end{equation}%
and the Riemann
curvature tensor $R$ and the Ricci tensor $S$ satisfy the following
equations \cite{Tri-KYK-10}:%
\begin{equation}
R\left( X,Y\right) \xi =\eta \left( X\right) Y-\eta \left( Y\right) X,
\label{eq-eps-PS-R(X,Y)xi}
\end{equation}%
\begin{equation}
R\left( \xi ,X\right) Y=-\,\varepsilon g\left( X,Y\right) \xi +\eta \left(
Y\right) X,  \label{eq-eps-PS-R(xi,X)Y}
\end{equation}%
\begin{equation}
\eta \left( R\left( X,Y\right) Z\right) =-\,\varepsilon \eta \left( X\right)
g\left( Y,Z\right) +\varepsilon \eta \left( Y\right) g\left( X,Z\right),
\label{eq-eps-PS-eta(R(X,Y),Z)}
\end{equation}%
\begin{equation}
S(X,\xi )=-(n-1)\eta (X),  \label{eq-eps-PS-S(X,xi)}
\end{equation}
for any $X,Y,Z\in \Gamma (TM)$.

\bigskip

Also remark that if $(\varphi,\xi,\eta,g,\varepsilon)$ is an $\left( \varepsilon \right) $-para Sasakian structure on the manifold $(M,g)$ of constant curvature $k$, if $M$ is spacelike (resp. timelike), then $M$ is hyperbolic (resp. elliptic) manifold. Indeed, if $R(X,Y)Z=k[g(Y,Z)X-g(X,Z)Y]$, for any $X$, $Y$, $Z\in \Gamma(TM)$, applying $\eta$ to this relation and using (\ref{eq-eps-PS-eta(R(X,Y),Z)}) we obtain $k=-\varepsilon$.

\bigskip

\begin{example}
\cite{Tri-KYK-10} Let ${\Bbb R}^{5}$\ be the $5$-dimensional real number
space with a coordinate system $\left( x,y,z,t,s\right) $. Defining
\[
\eta =ds-ydx-tdz\ ,\qquad \xi =\frac{\partial }{\partial s}\,
\]%
\[
\varphi \left( \frac{\partial }{\partial x}\right) =-\,\frac{\partial }{%
\partial x}-y\frac{\partial }{\partial s}\ ,\qquad \varphi \left( \frac{%
\partial }{\partial y}\right) =-\,\frac{\partial }{\partial y}\,
\]%
\[
\varphi \left( \frac{\partial }{\partial z}\right) =-\,\frac{\partial }{%
\partial z}-t\frac{\partial }{\partial s}\ ,\qquad \varphi \left( \frac{%
\partial }{\partial t}\right) =-\,\frac{\partial }{\partial t}\ ,\qquad
\varphi \left( \frac{\partial }{\partial s}\right) =0\,
\]%
\[
g_{1}=\left( dx\right) ^{2}+\left( dy\right) ^{2}+\left( dz\right)
^{2}+\left( dt\right) ^{2}-\eta \otimes \eta \,
\]%
\begin{eqnarray*}
g_{2} &=&-\,\left( dx\right) ^{2}-\left( dy\right) ^{2}+\left( dz\right)
^{2}+\left( dt\right) ^{2}+\left( ds\right) ^{2} \\
&&-\,t\left( dz\otimes ds+ds\otimes dz\right) -y\left( dx\otimes
ds+ds\otimes dx\right),
\end{eqnarray*}%
then $(\varphi ,\xi ,\eta ,g_{1})$ is a timelike Lorentzian almost
paracontact structure in ${\Bbb R}^{5}$, while $(\varphi ,\xi ,\eta
,g_{2})$\ is a spacelike $\left( \varepsilon \right) $-almost paracontact
structure.
\end{example}

\section{Almost $\eta$-Ricci solitons in $(M,\varphi ,\xi ,\eta ,g,\varepsilon )$}

Let $(M,g)$ be an $n$-dimensional Riemannian manifold ($n>2$). Consider the equation:
\begin{equation}\label{e8}
\pounds_{\xi}g+2S+2\lambda g+2\mu\eta\otimes \eta=0,
\end{equation}
where $\pounds_{\xi}$ is the Lie derivative operator along the vector field $\xi$, $S$ is the Ricci curvature tensor field of the metric $g$, $\eta$ is a $1$-form and $\lambda$ and $\mu$ are smooth functions on $M$.

The data $(\xi,\lambda,\mu)$ which satisfy the equation (\ref{e8}) is said to be an \textit{almost $\eta$-Ricci soliton} on $(M,g)$ called \textit{steady} if $\lambda=0$, \textit{shrinking} if $\lambda<0$ or \textit{expanding} if $\lambda>0$; in particular, if $\mu=0$, $(\xi,\lambda)$ is an \textit{almost Ricci soliton} \cite{pi}.

\bigskip

\begin{example}
Let $M=\mathbb{R}^3$, $(x,y,z)$ be the standard coordinates in $\mathbb{R}^3$ and $g$ be the Lorentzian metric:
$$g:=e^{-2z}dx\otimes dx+e^{2x-2z}dy\otimes dy-dz\otimes dz.$$

Consider the $(1,1)$-tensor field
$\varphi$, the vector field $\xi$ and the $1$-form $\eta$:
$$\varphi:=-\frac{\partial}{\partial x}\otimes dx-\frac{\partial}{\partial y}\otimes dy, \ \ \xi:=\frac{\partial}{\partial z}, \ \ \eta:=dz.$$

In this case, $(M,\varphi, \xi,\eta,g)$ is a timelike Lorentzian almost paracontact manifold.
Moreover, for the orthonormal vector fields:
$$E_1=e^z\frac{\partial}{\partial x}, \ \ E_2=e^{z-x}\frac{\partial}{\partial y}, \ \ E_3=\frac{\partial}{\partial z},$$
we get:
$$\nabla_{E_1}E_1=-E_3, \ \ \nabla_{E_1}E_2=0, \ \ \nabla_{E_1}E_3=-E_1, \ \ \nabla_{E_2}E_1=e^zE_2, \ \ \nabla_{E_2}E_2=-e^zE_1-E_3,$$$$\nabla_{E_2}E_3=-E_2, \ \ \nabla_{E_3}E_1=0, \ \ \nabla_{E_3}E_2=0, \ \ \nabla_{E_3}E_3=0$$
and the Riemann and the Ricci curvature tensor fields are given by:
$$R(E_1,E_2)E_2=(1-e^{2z})E_1, \ \ R(E_1,E_3)E_3=-E_1, \ \ R(E_2,E_1)E_1=(1-e^{2z})E_2,$$
$$R(E_2,E_3)E_3=-E_2, \ \ R(E_3,E_1)E_1=E_3, \ \ R(E_3,E_2)E_2=E_3,$$
$$S(E_1,E_1)=S(E_2,E_2)=2-e^{2z}, \ \ S(E_3,E_3)=-2.$$

Therefore, the data $(\xi,\lambda,\mu)$ for $\lambda=e^{2z}-1$ and $\mu=e^{2z}+1$ defines an almost $\eta$-Ricci soliton on $(M,g)$.
\end{example}

\bigskip

Writing $\pounds_{\xi}g$ in terms of the Levi-Civita connection $\nabla$, we obtain:
\begin{equation}\label{e9}
S(X,Y)=
-\frac{1}{2}[g(\nabla_X\xi,Y)+g(X,\nabla_Y\xi)]-\lambda g(X,Y)-\mu\eta(X)\eta(Y),
\end{equation}
for any $X$, $Y\in \Gamma(TM)$.

If $(M,\varphi ,\xi ,\eta ,g,\varepsilon )$ is an $\left( \varepsilon \right) $-para Sasakian manifold, then (\ref{e9}) becomes:
\begin{equation}\label{ea9}
S(X,Y)=
-\varepsilon g(\varphi X,Y)-\lambda g(X,Y)-\mu\eta(X)\eta(Y),
\end{equation}
for any $X$, $Y\in \Gamma(TM)$.

\bigskip

Also remark that on an $\left( \varepsilon \right) $-para Sasakian manifold $(M,\varphi ,\xi ,\eta ,g,\varepsilon )$, since the Ricci curvature tensor field $S$ satisfies (\ref{eq-eps-PS-S(X,xi)}), using (\ref{ea9}) we obtain:
\begin{equation}\label{ad}
\varepsilon \lambda+\mu=n-1
\end{equation}
and we can state:

\begin{proposition}
The scalar curvature of an $\left( \varepsilon \right) $-para Sasakian manifold $(M,\varphi ,\xi ,\eta ,g,\varepsilon )$ admitting an almost $\eta$-Ricci soliton $(\xi,\lambda,\mu)$ is:
\begin{equation}\label{c}
scal=-div(\xi)-n\lambda - \varepsilon \mu.
\end{equation}
\end{proposition}

From (\ref{ad}) we also deduce that:
\begin{proposition}
An almost Ricci soliton $(\xi,\lambda)$ on a spacelike (resp. timelike) $\left( \varepsilon \right) $-para Sasakian manifold $(M,\varphi ,\xi ,\eta ,g,\varepsilon )$ is expanding (resp. shrinking).
\end{proposition}

\begin{remark}
If the almost $\eta$-Ricci soliton is steady, then $\mu=n-1$ and the scalar curvature of $(M,\varphi ,\xi ,\eta ,g,\varepsilon )$ is $scal=-div(\xi)-\varepsilon (n-1)$.
\end{remark}

\bigskip

Similar like in the case of $\eta$-Ricci solitons on Lorentzian para-Sasakian manifolds \cite{bl}, the next theorems formulate some results in case of $\left( \varepsilon \right) $-para Sasakian manifold which is Ricci symmetric, has $\eta$-parallel, Codazzi or cyclic $\eta$-recurrent Ricci curvature tensor.

\begin{proposition}\label{t1}
Let $(\varphi ,\xi ,\eta ,g,\varepsilon )$ be an $\left( \varepsilon \right) $-para Sasakian structure on the manifold $M$ and let $(\xi,\lambda,\mu)$ be an almost $\eta$-Ricci soliton on $(M,g)$.
\begin{enumerate}
  \item If the manifold $(M,g)$ is Ricci symmetric (i.e. $\nabla S=0$), then $\xi(\mu)=-\varepsilon \xi(\lambda)$.
  \item If the Ricci tensor is $\eta$-recurrent (i.e. $\nabla S=\eta\otimes S$), then $\xi(\varepsilon \lambda+\mu)=n-1$.
  \item If the Ricci tensor is Codazzi (i.e. $(\nabla_X S)(Y,Z)=(\nabla_Y S)(X,Z)$, for any $X$, $Y$, $Z\in \Gamma(TM)$), then $d(\varepsilon \lambda+\mu)=\xi(\varepsilon \lambda+\mu)\eta$.
      \item If the Ricci tensor is $\eta$-parallel (i.e.$(\nabla_X S)(\varphi Y,\varphi Z)=0$), then the scalar function $\lambda$ is locally constant.
      \end{enumerate}
\end{proposition}
\begin{proof}
Replacing the expression of $S$ from (\ref{ea9}) in $(\nabla_XS)(Y,Z):=X(S(Y,Z))-S(\nabla_XY,Z)-S(Y,\nabla_XZ)$ we obtain:
     \begin{equation}\label{m}
     (\nabla_XS)(Y,Z)=\eta(Y)g(X,Z)+\eta(Z)g(X,Y)-2\varepsilon \eta(X)\eta(Y)\eta(Z)-\end{equation}$$-X(\lambda)g(Y,Z)-X(\mu)\eta(Y)\eta(Z)-\mu[\eta(Y)g(\varphi X,Z)+\eta(Z)g(\varphi X,Y)].$$

For the first two assertions, just take $X=Y=Z:=\xi$ in the expression of $\nabla S$ from (\ref{m}) and we obtain the required results. Concerning the case when $S$ is Codazzi, taking $Y=Z:=\xi$ in $(\nabla_X S)(Y,Z)=(\nabla_Y S)(X,Z)$ we obtain $\varepsilon X(\lambda)+X(\mu)=[\varepsilon \xi(\lambda)+\xi(\mu)]\eta(X)$, for any $X\in \Gamma(TM)$, which is equivalent to the stated relation. If $S$ is $\eta$-parallel, then for any $X$, $Y$, $Z\in \Gamma(TM)$, $0=(\nabla_X S)(\varphi Y,\varphi Z)=-X(\lambda)g(\varphi Y,\varphi Z)$ which implies $X(\lambda)=0$, for any $X\in \Gamma(TM)$.
\end{proof}

\bigskip

\begin{remark}
If on the $\left( \varepsilon \right) $-para Sasakian manifold $(M,\varphi ,\xi ,\eta ,g,\varepsilon )$ we consider the almost $\eta$-Ricci soliton $(V, \lambda, \mu)$ with the potential vector field $V$ conformal Killing (i.e. $\frac{1}{2}\pounds_{\xi}g=fg$), for $f$ a smooth function on $M$, then
$$S=-(f+\lambda)g-[n-1-\varepsilon (f+\lambda)]\eta\otimes \eta$$
and the manifold is Einstein if and only if $\lambda=\varepsilon(n-1)-f$; in this case, we have an almost Ricci soliton.

The same conclusion is reached if the potential vector field $V:=\xi$ is torse-forming (i.e. $\nabla \xi =f\varphi ^{2}$ according to \cite{per}).
\end{remark}

\bigskip

When the potential vector field of (\ref{e8}) is of gradient type, i.e. $\xi=grad(f)$, then $(\xi,\lambda,\mu)$ is said to be a \textit{gradient almost $\eta$-Ricci soliton} and the equation satisfied by it becomes:
\begin{equation}\label{e22}
Hess(f)+S+\lambda g+\mu\eta\otimes \eta=0,
\end{equation}
where $Hess(f)$ is the Hessian of $f$ defined by $Hess(f)(X,Y):=g(\nabla_X\xi,Y)$.

\begin{proposition}
Let $(M,\varphi,\xi,\eta,g,\varepsilon)$ be an $\left( \varepsilon \right) $-para Sasakian manifold. If (\ref{e22}) defines a gradient almost $\eta$-Ricci soliton on $(M,g)$
with the potential vector field $\xi:=grad(f)$ and $\eta=df$ is the $g$-dual of $\xi$, then:
\begin{equation}
(\nabla_XQ)Y-(\nabla_YQ)X=(d(f-\lambda)\otimes I-I\otimes d(f-\lambda)+(df\otimes d\mu-d\mu\otimes df)\otimes \xi+\end{equation}$$+\varepsilon \mu(df\otimes \varphi-\varphi\otimes df))(X,Y),
$$
for any $X$, $Y\in\Gamma(TM)$, where $Q$ stands for the Ricci operator defined by $g(QX,Y):=S(X,Y)$.
\end{proposition}
\begin{proof}
Notice that (\ref{e22}) can be written:
\begin{equation}\label{e23}
\nabla\xi+Q+\lambda I+\mu df\otimes \xi=0.
\end{equation}

Then:
\begin{equation}
(\nabla_XQ)Y=-(\nabla_X\nabla_Y\xi-\nabla_{\nabla_XY}\xi)-X(\lambda)Y-X(\mu)df(Y)\xi-\mu[g(Y,\nabla_X\xi)\xi+df(Y)\nabla_X\xi].
\end{equation}

Replacing now $\nabla\xi=\varepsilon \varphi$ in the previous relation, after a long computation, we get the required relation.
\end{proof}

\begin{remark}
i) Remark that since $\xi$ is geodesic vector field, from (\ref{e23}) follows that $\xi$ is an eigenvector of $Q$ corresponding to the eigenvalue $-(\lambda+\mu)$. In particular, if $\lambda=-\mu$, then $\xi\in \ker Q$.

ii) The Ricci operator is $\varphi$-invariant (i.e. $Q\circ \varphi=\varphi \circ Q$).

iii) If $Q$ is Codazzi (i.e. $(\nabla_X Q)Y=(\nabla_Y Q)X$, for any $X$, $Y\in \Gamma(TM)$), then for the gradient $\eta$-Ricci soliton case, $\mu \varphi X=\varepsilon \varphi ^2X$, for any $X\in \Gamma(TM)$ which implies $\mu^2=1$. Therefore, the soliton is expanding if $M$ is spacelike or shrinking if $M$ is timelike and it is given by $(\lambda,\mu)\in \{ (\varepsilon n, -1), (\varepsilon (n-2), 1)\}$.
\end{remark}

\bigskip

A lower and an upper bound of the Ricci curvature tensor's norm \cite{cr} for a gradient almost $\eta$-Ricci soliton on an $\left( \varepsilon \right) $-para Sasakian manifold will be given \cite{bla}.

\begin{theorem}
If (\ref{e22}) defines a gradient almost $\eta$-Ricci soliton on the $n$-dimensional $\left( \varepsilon \right) $-para Sasakian manifold $(M,\varphi,\xi,\eta,g,\varepsilon)$ and
$\eta=df$ is the $g$-dual of the gradient vector field $\xi:=grad(f)$, then:
\begin{equation}\label{e21}
n-1+\mu^2-\frac{(\Delta(f)+\varepsilon\mu)^2}{n}\leq |S|^2\leq n-1+\mu^2+\frac{(scal)^2}{n}.
\end{equation}
\end{theorem}

\begin{remark}
The simultaneous equalities hold for $(scal)^2=-(\Delta(f)+\varepsilon\mu)^2 \ (=0)$ i.e. for steady gradient almost $\eta$-Ricci soliton ($\lambda=0$) with $scal=0$ and $\Delta(f)=-\varepsilon\mu$. In this case, $|S|^2=n-1+\mu^2$.
\end{remark}

\bigskip

Using a Bochner-type formula for the gradient almost $\eta$-Ricci solitons \cite{bla}, in the case of $\left( \varepsilon \right) $-para Sasakian manifold we obtain the condition satisfied by the potential function:

\begin{theorem}
If (\ref{e22}) defines a gradient almost $\eta$-Ricci soliton on the $n$-dimensional $\left( \varepsilon \right) $-para Sasakian manifold $(M,\varphi,\xi,\eta,g,\varepsilon)$ and
$\eta=df$ is the $g$-dual of the gradient vector field $\xi:=grad(f)$, then:
\begin{equation}\label{e1}
\Delta(f)=\frac{1}{2}[\varepsilon(n-1)+\lambda +\varepsilon \mu+\varepsilon (n-2)\xi(\lambda)-\xi(\mu)].
\end{equation}
\end{theorem}

Considering (\ref{c}) and (\ref{e1}) we obtain:
\begin{equation}\label{e}
scal=-\frac{1}{2}[\varepsilon(n-1)+(2n+1)\lambda+3\varepsilon\mu+\varepsilon(n-2)\xi(\lambda)-\xi(\mu)].
\end{equation}

\begin{remark}
i) If (\ref{e22}) defines a gradient $\eta$-Ricci soliton on the $n$-dimensional $\left( \varepsilon \right) $-para Sasakian manifold $(M,\varphi,\xi,\eta,g,\varepsilon)$, then:
$$scal=-\frac{1}{2}[\varepsilon(n-1)+(2n+1)\lambda+3\varepsilon\mu].$$

ii) In this case, if $M$ is connected and has constant scalar curvature, then:
$$\mu=-\frac{2n+1}{3\varepsilon}\lambda+C, \ \ C\in \mathbb{R}.$$
\end{remark}

\textit{Adara M. Blaga}

\textit{Department of Mathematics}

\textit{West University of Timi\c{s}oara}

\textit{Bld. V. P\^{a}rvan nr. 4, 300223, Timi\c{s}oara, Rom\^{a}nia}

\textit{adarablaga@yahoo.com}

\bigskip

\textit{Selcen Y\"{u}ksel Perkta\c{s}}

\textit{Department of Mathematics, Faculty of Arts and Sciences}

\textit{Ad\i yaman University}

\textit{02040, Ad\i yaman, Turkey}

\textit{sperktas@adiyaman.edu.tr}

\end{document}